\documentclass[12pt]{article}
\usepackage{amsmath,amssymb, amsfonts, amsthm,amscd, dsfont}
\usepackage[T2A]{fontenc}
\usepackage[cp1251]{inputenc}
\usepackage[english]{babel}
\usepackage{geometry}

\newtheorem{proposition}{Proposition}
\newtheorem{lemma}{Lemma}
\newtheorem{theorem}{Theorem}
\newtheorem{corollary}{Corollary}
\newtheorem*{proposition*}{Proposition}
\newtheorem*{lemma*}{Lemma}
\newtheorem*{theorem*}{Theorem}
\newtheorem*{corollary*}{Corollary}

\newtheorem*{definition*}{Definition}
\newtheorem*{remark*}{Remark}
\newtheorem*{example*}{Example}
\input{tcilatex}

\begin{document}

\begin{center}
{\LARGE On modules with few minimax cocentralizers}

Leonid~A.~Kurdachenko$^1$, Igor~Ya.~Subbotin$^2$ and Vasiliy A.~Chupordya$^1$

$^1$ Department of Algebra and Geometry, Dnipropetrovsk National University,
Gagarin prospect 72, Dnepropetrovsk 10, 49010, Ukraine

$^2$ Mathematics Department, National University, 5245 Pacific Concourse
Drive, Los Angeles, CA 90045, USA
\end{center}

\textbf{Abstract:} Let $R$ be a ring and $G$ a group. An $R$-module $A$ is
said to be minimax if $A$ includes an noetherian submodule $B$ such that $A/B
$ is artinian. The authors study a $\mathbb{Z}G$-module A such that $A/C_A(H)
$ is minimax (as a $\mathbb{Z}$-module) for every proper not finitely
generated subgroup $H$.

\textbf{Keywords:} minimax modules, minimax cozentralizers, finitary modules
over group rings, minimax-antifinitary $RG$-module, generalized radical
groups.

Classification: 20C07, 20F19.

\section{Introduction}

The modules over group rings $RG$ are classical objects of study with well
established links to various areas of algebra. The case when $G$ is a finite
group has been studying in sufficient details for a long time. For the case
when $G$ is an infinite group, the situation is different. The investigation
of modules over polycyclic-by-finite groups was initiated in the classical
works of P.~Hall~\cite{HP1954, HP1959}. Nowadays, the theory of modules over
polycyclic-by-finite groups is highly developed and rich on interesting
results. This was largely due to the fact that a group ring $RG$ of a
polycyclic-by-finite group $G$ over a noetherian ring $R$ is also
noetherian. The group rings over some other groups (even over well-studied
groups, as for instance, the Chernikov groups) do not always have such good
properties as, for example, being noetherian. Therefore, their investigation
requires some different approaches and restrictions. The classical
finiteness conditions are widely popular such kind of restrictions. The very
first restrictions here were those that came from ring theory, namely the
conditions "to be noetherian" and "to be artinian". Noetherian and artinian
modules over group rings are also very well investigated. Many aspects of
the theory of artinian modules over group rings are well reflected in the
book~\cite{KOS2007}. Lately the so-called finitary approach is under
intensive development. This is mainly due to the progress which its
applications have found in the theory of infinite dimensional linear groups.

Let $R$ be a ring, $G$ a group and $A$ an $RG$-module. For a subgroup $H$ of 
$G$ we consider the $R$-submodule $C_A(H)$. Then $H$ acts on $A/C_A(H)$. The 
$R$-factor-module $A/C_A(H)$ is called the \textit{cocentralizer of $H$ in $A
$}. The factor-group $H/C_H(A/C_A(H))$ is isomorphic to a subgroup of
automorphisms group of an $R$-module $A/C_A(H)$. If $x$ is an element of $%
C_H(A/C_A(H))$, then $x$ acts trivially on factors of the series $%
\langle0\rangle \leq C_A(H) \leq A.$ It follows that $C_H(A/C_A(H))$ is
abelian. This shows that the structure of $H$ to a greater extent is defined
by the structure of $C_H(A/C_A(H))$, and hence by the structure of the
automorphisms group of the $R$-module $A/C_A(H)$. Let $\mathfrak{M}$ be a
class of $R$-modules. We say that $A$ is \textit{$\mathfrak{M}$-finitary
module over $RG$} if $A/C_A(x) \in \mathfrak{M}$ for each element $x \in G$.
If $R$ is a field, $C_G(A) = \langle1\rangle$, and $\mathfrak{M}$ is the
class of all finite dimensional vector spaces over $R$, then we come to the
finitary linear groups. The theory of finitary linear groups is quite well
developed (see, for example, the survey~\cite{PR1995}). B.A.F.~Wehrfritz
began considering the cases when $\mathfrak{M}$ is the class of finite $R$%
-modules~\cite{WB2002[1], WB2002[3], WB2002[4], WB2004[1]}, when $\mathfrak{M%
}$ is the class of noetherian $R$-modules~\cite{WB2002[2]}, when $\mathfrak{M%
}$ is the class of artinian $R$-modules~\cite{WB2002[4], WB2003, WB2004[1],
WB2004[2], WB2005}. The artinian-finitary modules have been considered also
in the paper~\cite{KSC2007}. The artinian and noetherian modules can be
united into the following type of modules. An $R$-module $A$ is said to be 
\textit{minimax} if $A$ has a finite series of submodules, whose factors are
either noetherian or artinian. It is not hard to show that if $R$ is an
integral domain, then every minimax $R$-module $A$ includes a noetherian
submodule $B$ such that $A/B$ is artinian. The first natural case here is
the case when $R=\mathbb{Z}$ is the ring of all integers. B.A.F.~Wehrfritz
has began the study of noetherian-finitary and artinian-finitary modules
with separate consideration of this case. This case is of particular
importance in applications, for instance, it is very important in the theory
of generalized soluble groups.

Let $G$ be a group, $A$ an $RG$-module, and $\mathfrak{M}$ a class of $R$%
-modules. Put

\begin{equation*}
\mathcal{C}_\mathfrak{M}(G) = \{ H \ | \ H \mbox{ is a subgroup of } G %
\mbox{ such that } A/C_A(H) \in \mathfrak{M} \}.
\end{equation*}

If $A$ is an $\mathfrak{M}$-finitary module, then $\mathcal{C}_\mathfrak{M}%
(G)$ contains every cyclic subgroup (moreover, every finitely generated
subgroup whenever $\mathfrak{M}$ satisfies some natural restrictions). It is
clear that the structure of $G$ depends significantly on which important
subfamilies of the family $\Lambda(G)$ of all proper subgroups of $G$
include $\mathcal{C}_\mathfrak{M}(G)$. Therefore it is interesting to
consider the cases when the family $\mathcal{C}_\mathfrak{M}(G)$ is large.
The case, when $\mathcal{C}_\mathfrak{M}(G)$ is a family of all proper
subgroups was discussed in another paper by the same authors. In almost all
groups (with exception of noetherian groups), the family of subgroups which
is not finitely generated is much larger than the family of finitely
generated subgroups. It is therefore interesting to consider the case, which
is dual to the case of an $\mathfrak{M}$-finitary module.

Let $R$ be a ring, $G$ be a group and $A$ an $RG$-module. We say that $A$ is 
\textit{minimax-antifinitary $RG$-module} if the factor-module $A/C_A(H)$ is
minimax as an R-module for each not finitely generated proper subgroup $H$
and the $R$-module $A/C_A(G)$ is not minimax.

This current work is devoted to the study of a minimax-antifinitary $\mathbb{%
Z}G$-modules where $G$ belongs to the following very large class of groups.

A group $G$ is called \textit{generalized radical} if $G$ has an ascending
series whose factors are locally nilpotent or locally finite. Hence a
generalized radical group $G$ either has an ascendant locally nilpotent
subgroup or an ascendant locally finite subgroup. In the first case, the
locally nilpotent radical $\mathbf{Lnr}(G)$ of $G$ is non-identity. In the
second case, it is not hard to see that $G$ includes a non-identity normal
locally finite subgroup. Clearly in every group G the subgroup $\mathbf{Lfr}%
(G)$ generated by all normal locally finite subgroups is the largest normal
locally finite subgroup (the \textit{locally finite radical}). Thus every
generalized radical group has an ascending series of normal subgroups with
locally nilpotent or locally finite factors.

The study breaks down naturally into the following cases. The first case is
the case when $G = \mathbf{Coc}_{\mathbb{Z}-mmx(G)} = \{ x \ | \ A/C_A(x) %
\mbox{ is a minimax }\mathbb{Z}-\mbox{module} \}$. In this case, every
proper subgroup of $G$ has a minimax cocentralizer. As we noted above, this
case was considered separately in another paper. The second case is the case
when $G \neq \mathbf{Coc}_{\mathbb{Z}-mmx}(G)$ and the group $G$ is not
finitely generated. The third case is the case when $G \neq \mathbf{Coc}_{%
\mathbb{Z}-mmx}(G)$ and the group $G$ is finitely generated. The current
article is dedicated to the second case. Its main result is the following

\begin{theorem}
Let $G$ be a locally generalized radical group, $A$ a minimax-antifinitary $%
\mathbb{Z}G$-module, and $D = \mathbf{Coc}_{\mathbb{Z}-mmx}(G)$. Suppose
that $G$ is not finitely generated, $G \neq D$ and $C_G(A) = \langle1\rangle$%
. Then $G$ is a group of one of the following types:

\begin{enumerate}
\item $G$ is a quasicyclic $q$-group for some prime $q$. 

\item $G = Q \times \langle g \rangle$ where $Q$ is a quasicyclic $q$%
-subgroup, $g$ is a $p$-element and $g^p \in D$, $p$, $q$ are prime (not
necessary different). 

\item $G$ includes a normal divisible Chernikov $q$-subgroup $Q$, such that $%
G = Q\langle g \rangle$ where $g$ is a $p$-element, $p$, $q$ are prime (not
necessary different). Moreover, $G$ satisfies the following conditions: 

\begin{enumerate}
\item \label{3a} $g^p \in D$; 

\item \label{3b} $Q$ is $G$-quasifinite; 

\item \label{3c} If $q = p$, then $Q$ has special rank $p^{m-1}(p-1)$ where $%
p^m =| \langle g\rangle/C_{\langle g \rangle}(Q)|$; 

\item \label{3d} if $q \neq p$, then $Q$ has special rank $\mathbf{o}(q, p^m)
$ where again $p^m = | \langle g \rangle/C_{\langle g \rangle}(Q)|$ and $%
\mathbf{o}(q, p^m)$ is the order of $q$ modulo $p^m$. 
\end{enumerate}
\end{enumerate}

Furthermore, for the types 2, 3 $A(\omega\mathbb{Z}D)$ is a Chernikov
subgroup and $\Pi(A(\omega\mathbb{Z}D)) \subseteq \Pi(D)$.
\end{theorem}

Here $\omega RG$ be the \textit{augmentation ideal} of the group ring $RG$,
the two-sided ideal of $RG$ generated by all elements $g-1$, $g \in G$.

Recall also that an abelian normal subgroup $A$ of a group $G$ is called 
\textit{$G$-quasifinite} if every proper $G$-invariant subgroup of $A$ is
finite. Clearly that in this case either $A$ is a union of its finite $G$%
-invariant subgroups or $A$ includes a finite $G$-invariant subgroup $B$
such that the factor $A/B$ is $G$-chief. At the end of the article, we
provide the examples showing that all the situations that arise in the
theorem can be realized.

\section{Some preliminary results}

Let $R$ be a ring and $\mathfrak{M}$ a class of $R$-modules. Then $\mathfrak{%
M}$ is said to be a \textit{formation} if it satisfies a following
conditions:

F1 if $A \in \mathfrak{M}$ and $B$ is an $R$-submodule of $A$, then $A/B \in 
\mathfrak{M}$;

F2 if $A$ is an $R$-module and $B_1, ..., B_k$ are $R$-submodules of $A$
such that $A/B_j \in \mathfrak{M}$, $1 \leq j \leq k$, then $A/( B_1 \cap
... \cap B_k) \in \mathfrak{M}$.

\begin{lemma}
\label{L1} Let $R$ be a ring, $\mathfrak{M}$ a formation of $R$-modules, $G$
a group and $A$ an $RG$-module.

(i) If $L, H$ are subgroups of $G$ such that $L \leq H$ and $A/C_A(H) \in 
\mathfrak{M}$, then ${A/C_A(L) \in \mathfrak{M}}$.

(ii) If $L, H$ are subgroups of $G$ whose cocentralizers belong to $%
\mathfrak{M}$, then ${A/C_A(\langle H, L \rangle) \in \mathfrak{M}}$.
\end{lemma}

\begin{proof}
The inclusion $L \leq H$ implies that $C_A(L) \geq C_A(H)$. Since $A/C_A(H) \in \mathfrak{M}$ and $\mathfrak{M}$ is a formation, $A/C_A(L) \in \mathfrak{M}$. Clearly $C_A(\langle H, L\rangle) \leq C_A(H) \cap C_A(L)$. Since $\mathfrak{M}$ is a formation, $A/(C_A(H) \cap C_A(L)) \in \mathfrak{M}$. Then and $A/C_A(\langle H, L \rangle) \in \mathfrak{M}$.
\end{proof}

\begin{lemma}
Let $R$ be a ring, $\mathfrak{M}$ a formation of $R$-modules, $G$ a group
and $A$ an $RG$-module. Then 
\begin{equation*}
\mathbf{Coc}_\mathfrak{M}(G) = \{ x \in G \ | \ A/C_A(x) \in \mathfrak{M} \}
\end{equation*}
is a normal subgroup of $G$.
\end{lemma}

\begin{proof}
By Lemma \ref{L1}, $\mathbf{Coc}_\mathfrak{M}(G)$ is a subgroup of $G$. Let $x \in \mathbf{Coc}_\mathfrak{M}(G)$, $g \in G$. Then $C_A(x^g) = C_A(x)g$. Since the mapping $a \mapsto  ag$, $a \in A$, is $R$-linear,
$$A/C_A(x) \cong_R Ag/C_A(x)g = A/C_A(x)g = A/C_A(x^g),$$
which shows that $A/C_A(x^g) \in \mathfrak{M}$, and hence $x^g \in \mathbf{Coc}_\mathfrak{M}(G)$.
\end{proof}

Clearly the class of minimax modules over an integral domain $R$ is a
formation, so we obtain

\begin{corollary}
\label{C3} Let $R$ be a ring, $G$ a group and $A$ an $RG$-module.

(i) $L$, $H$ are subgroups of $G$ such that $L \leq H$ and a factor-module $%
A/C_A(H)$ is minimax, then $A/C_A(L)$ is also minimax.

(ii) If $L$, $H$ are subgroups of $G$ whose cocentralizers are minimax, then 
$A/C_A(\langle H, L \rangle)$ is also minimax.
\end{corollary}

\begin{corollary}
Let $R$ be a ring, $G$ a group and $A$ an $RG$-module. Then 
\begin{equation*}
\mathbf{Coc}_{R-mmx}(G) = \{ x \in G \ | \ A/C_A(x) \mbox{ is minimax  }\}
\end{equation*}
is a normal subgroup of $G$.
\end{corollary}

A group $G$ is said to be \textit{$\mathfrak{F}$-perfect } if $G$ does not
include proper subgroups of finite index.

\begin{lemma}
\label{L5} Let $G$ be a locally generalized radical group and $A$ be a $%
\mathbb{Z}G$-module. Suppose that $A$ includes a $\mathbb{Z}G$-submodule $B$%
, which is minimax. Then the following assertions hold:

(i) $G/C_G(B)$ is soluble-by-finite.

(ii) If $G/C_G(B)$ is periodic, then it is nilpotent-by-finite.

(iii) If $G/C_G(B)$ is $\mathfrak{F}$-perfect and periodic, then it is
abelian, moreover ${\ [[B, G], G] = \langle0\rangle }.$
\end{lemma}

\begin{proof}

Without loss of generality we can suppose that $C_G(B) = \langle1\rangle$. Being minimax, $B$ has a series of $G$-invariant subgroups $\langle0\rangle \leq  D \leq K \leq B$  where $D$ is a divisible Chernikov subgroup, $K/D$ is finite and $B/K$ is torsion-free and has finite $\mathbb{Z}$-rank.  More exactly, $D = D_1 \oplus ... \oplus D_n$ where $D_j$ is a Sylow $p_j$-subgroup of $D$, $1 \leq j \leq n$. Clearly every subgroup $D_j$ is  $G$-invariant, $1 \leq j \leq n$. Let $q = p_j$. The factor-group $G/C_G(D_j)$  is isomorphic to the subgroup of $\mathbf{GL}_m(\mathbb{Z}_{q^\infty})$  where $\mathbb{Z}_{q^\infty}$ is the ring of integer $q$-adic numbers and $m$ satisfies $q^m = | \Omega_1(D_j)|$. Let $F$ be a field of fractions of  $\mathbb{Z}_{q^\infty}$, then $G/C_G(D_j)$ is isomorphic to a subgroup of $\mathbf{GL}_m(F)$. Note that $\mathbf{char}(F) = 0$. Being locally generalized radical, $G/C_G(D_j)$ does not include the non-cyclic free subgroup, thus application of Tits's theorem (see, for example,~\cite[Corollary 10.17]{WB1973}) shows that $G/C_G(D_j)$ is soluble-by-finite. If $G$ is periodic, then $G/C_G(D_j)$ is finite (see, for example,~\cite[Theorem 9.33]{WB1973}). It is valid for each $j$, $1 \leq j \leq n$. We have $C_G(D) = \bigcap_{1 \leq j \leq n}C_G(D_j)$. Therefore using Remak's theorem we obtain the imbedding
$$G/C_G(D) \hookrightarrow \mathbf{Dr}_{1 \leq j \leq n}G/C_G(D_j),$$
which shows that $G/C_G(D)$ is also soluble-by-finite (respectively finite). Since $K/D$ is finite, $G/C_G(K/D)$ is finite. Finally, $G/C_G(B/K)$ is isomorphic to a subgroup of $\mathbf{GL}_\mathbf{r}(\mathbb{Q})$, where $\mathbf{r} = \mathbf{r}_\mathbb{Z}(B/K)$. Using again the fact that  $G/C_G(A/K)$ does not include the non-cyclic free subgroup and Tits's theorem (respectively Theorem 9.33 of the book~\cite{WB1973}), we obtain that  $G/C_G(B/K)$ is soluble-by-finite (respectively finite). Put
$$Z = C_G(D) \cap C_G(K/D) \cap C_G(B/K).$$
Then  $G/Z$ is embedded in $G/C_G(D) \times G/C_G(K/D) \times G/C_G(B/K)$, in particular, $G/Z$ is soluble-by-finite (respectively finite). If $x \in Z$, then $x$ acts trivially in every factors of the series  $\langle0\rangle \leq D \leq K \leq A$. By Kaloujnin's theorem~\cite{KL1953} $Z$ is nilpotent. It follows that $G$ is soluble-by-finite (respectively nilpotent-by-finite).

Suppose now that $G$ is an $\mathfrak{F}$-perfect group. Again consider the series of $G$-invariant subgroups $\langle0\rangle \leq K \leq B$. Being abelian and Chernikov, $K$ is a union of the ascending series

$$\langle0\rangle = K_0 \leq K_1 \leq  ... \leq K_n \leq K_{n + 1} \leq ... $$
of $G$-invariant finite subgroups $K_n$, $n \in \mathbb{N}$. Then the factor-group $G/C_G(K_n)$ is finite, $n \in \mathbb{N}$. Since $G$ is  $F$-perfect, $G = C_G(K_n)$ for each $n \in \mathbb{N}$. The equation $K = \bigcup_{n \in \mathbb{N}} K_n$ implies that $G = C_G(K)$. By proved above,  $G/C_G(B/K)$  is soluble-by-finite, and being $\mathfrak{F}$-perfect, is soluble. Then $G/C_G(B/K)$ includes normal subgroups $U$, $V$  such that  $C_G(B/K) \leq U \leq V$, $U/C_G(B/K)$  is isomorphic to a subgroup of $\mathbf{UT}_\mathbf{r}(\mathbb{Q})$, $V/U$ includes a free abelian subgroup of finite index~\cite[Theorem 2]{CVS1954}. Since $G/C_G(B/K)$ is $\mathfrak{F}$-perfect, it follows that $G/C_G(B/K)$ is torsion-free. Being periodic, $G/C_G(B/K)$  must be identity. In other words, $G = C_G(B/K)$. Hence $G$ acts trivially in every factors of a series $\langle0\rangle \leq K \leq A$, so that $[[B, G], G] = \langle0\rangle$, and using again Kaloujnin's theorem~\cite{KL1953}, we obtain that $G$ is abelian.
\end{proof}

\begin{corollary}
\label{C6} Let $G$ be a group and $A$ a $\mathbb{Z}G$-module. If the
factor-module $A/C_A(G)$ is minimax as a $\mathbb{Z}$-module, then every
locally generalized radical subgroup of $G/C_G(A)$ is soluble-by-finite, and
every periodic subgroup of $G/C_G(A)$ is nilpotent-by-finite.
\end{corollary}

\begin{proof}
Indeed, Lemma~\ref{L5} shows that $G/C_G(A/C_A(G))$  is soluble-by-finite. Every element $x \in C_G(A/C_A(G))$  acts trivially in the factors of the series $\langle0\rangle \leq C_A(G) \leq A$. It follows that $C_G(A/C_A(G))$ is abelian. Suppose now that  $H/C_G(A)$  is a periodic subgroup. Since  $A/C_A(G)$  is minimax, $A$  has a series of  $H$-invariant subgroups $\langle0\rangle \leq C_A(G) \leq D \leq K \leq A$  where $D/C_A(G)$ is divisible Chernikov subgroup, $K/D$ is finite and $A/K$ is torsion-free and has finite $\mathbb{Z}$-rank.  In Lemma~\ref{L5} we have already proved that $G/C_G(D/C_A(G))$, $G/C_G(K/D)$ and $G/C_G(A/K)$ are finite.  Let  $Z = C_G(D/C_A(G)) \cap C_G(K/D) \cap C_G(A/K)$. Then  $G/Z$ is finite. If  $x \in Z$, then $x$ acts trivially in every factors of a series $\langle0\rangle \leq C_A(G) \leq D \leq K \leq A$. By Kaloujnin's theorem~\cite{KL1953} $Z$ is nilpotent.
\end{proof}

Let $G$ be a generalized radical group and let $R_1$ be a normal subgroup of 
$G$, satisfying the following conditions: $R_1$ is radical, $G/R_1$ does not
include the non-identity locally nilpotent normal subgroups. Then $G/R_1$
must include a non-identity normal locally finite subgroup. It follows that
the locally finite radical $R_2/R_1$ is non-identity. If we suppose that $%
G/R_2$ includes a non-identity normal locally finite subgroup $L/R_2$, then $%
L/R_1$ is also locally finite, which contradicts to the choice of $R_2$.
This contradiction shows that $G/R_2$ does not include a non-identity normal
locally finite subgroup, and therefore it must include a non-identity normal
locally nilpotent subgroup. Let $R_3/R_2$ be a normal subgroup of $G/R_2$,
satisfying the following conditions: $R_3/R_2$ is radical, $G/R_3$ does not
include non-identity locally nilpotent normal subgroups. Using similar
arguments, we construct the ascending series of normal subgroups 
\begin{equation*}
\langle1\rangle = R_0 \leq R_1 \leq \ldots R_{\alpha} \leq R_{\alpha + 1}
\leq \ldots \leq R_\gamma = G,
\end{equation*}
whose factors are radical or locally finite, and if $R_{\alpha + 1}/R_\alpha$
is radical (respectively locally finite), then $R_{\alpha + 2}/R_{\alpha + 1}
$ is locally finite (respectively radical).

This series is called a \textit{standard series} of a generalized radical
group $G$.

\begin{lemma}
\label{L7} Let $G$ be a group and $A$ an minimax-antifinitary $\mathbb{Z}G$%
-module. Then every proper generalized radical subgroup of $G/C_G(A)$ is
soluble-by-finite.
\end{lemma}

\begin{proof}

Again we will suppose that $C_G(A) = \langle1\rangle$. Let $L$ be an arbitrary proper generalized radical subgroup of  $G$. Let
$$\langle1\rangle = R_0 \leq R_1 \leq \ldots R_{\alpha} \leq  R_{\alpha + 1} \leq \ldots \leq R_\gamma = L,$$
be a standard series of $L$. Suppose that $\gamma\geq\omega$ ($\omega$ is the first infinite ordinal) and consider the subgroup $R_\omega$. Assume that $R_\omega$ is finitely generated, that is $R_\omega  = \langle u_1, \ldots u_t \rangle$ for some elements   $u_1, \ldots u_t$. The equation $R_\omega  = \bigcup_{n \in \mathbb{N}} R_n$  shows that there exists a positive integer $m$ such that  $u_1, \ldots u_t \in R_m$. But in this case, $R_\omega = R_m$ and we obtain a contradiction. This contradiction shows that $R_\omega$ is not finitely generated. It follows that $A/C_A(R_\omega)$ is minimax. Corollary~\ref{C6} shows that $R_\omega$ is soluble-by-finite.  In this case $R_\omega = R_2$ and we again obtain a contradiction. This contradiction shows that $\gamma$ must be finite, that  is $\gamma = k$ is some positive integer.

Now we will use induction by $k$. Consider the subgroup $R_1$. Then either $R_1$ is radical or $R_1$ is locally finite. If $R_1$ is not finitely generated, then $A/C_A(R_\omega)$ is minimax. Corollary~\ref{C6} shows that $R_1$ is soluble-by-finite. Suppose that $R_1$ is finitely generated. If $R_1$ is locally finite, then it is finite. Therefore  assume that $R_1$ is radical. Let
$$\langle1\rangle = V_0 \leq V_1 \leq \ldots V_{\alpha} \leq  V_{\alpha + 1} \leq \ldots \leq V_\eta = R_1,$$
be an ascending series of $R_1$  where $V_{\alpha + 1}/V_\alpha$  is the locally nilpotent radical of $R_1/V_\alpha$, $\alpha < \eta$. Using the above arguments we obtain that $\eta$ is finite, that  is $\eta = d$  is some positive integer. Let $m$ be a number such that all factors $V_{m + 1}/V_m, V_{m + 2}/V_{m + 1}, \ldots, V_d/V_{d-1}$ are finitely generated. Being locally nilpotent, they are polycyclic. It follows that $V_d/V_m$  is polycyclic. In particular if every subgroup $V_j$ is finitely generated, $1 \leq j \leq d$, then $R_1$ is polycyclic. Therefore assume that there is a positive integer $s$ such that $V_s$ is not finitely generated, but  a subgroup $V_j$ is finitely generated for all $j > s$. Then $A/C_A(V_s)$ is minimax and Corollary~\ref{C6} yields that $V_s$ is soluble. In this case $R_1/V_s$ is polycyclic, so that $R_1$ is soluble.

Suppose that we have already proved that all subgroups $R_1, R_2, \ldots, R_{k-1}$ are soluble-by-finite. Repeating the above arguments, we obtain that and $R_k$ is soluble-by-finite, and the result is proved.

\end{proof}

\begin{lemma}
\label{L8} Let $G$ be a group and $A$ an minimax-antifinitary $\mathbb{Z}G$%
-module. If $H$ is a proper subgroup of $G$ and $\mathbf{Coc}_{\mathbb{Z}%
-mmx}(G)$ does not include $H$, then $H$ is finitely generated.
\end{lemma}

\begin{proof}
Indeed if we suppose that $H$ is not finitely generated, then $A/C_A(H)$  is minimax. Corollary~\ref{C3} shows that $A/C_A(h)$ is minimax for each element $h \in H$. It follows that $H \leq  \mathbf{Coc}_{\mathbb{Z}-mmx}(G)$, and we obtain a contradiction with the choice of $H$.
\end{proof}

\section{Proofs of the main results.}

\begin{proposition}
\label{P2.1} Let $G$ be a locally generalized radical group and $A$ an
minimax-antifinitary $\mathbb{Z}G$-module. If $G/\mathbf{Coc}_{\mathbb{Z}%
-mmx}(G)$ is not finitely generated, then $G/C_G(A)$ is a quasicyclic $q$%
-group for some prime $q$.
\end{proposition}

\begin{proof}
Again suppose that $C_G(A) = \langle1\rangle$. Let $M = \mathbf{Coc}_{\mathbb{Z}-mmx}(G)$. Let $H$ be a proper subgroup of $G$. If $M$  does not include $H$, then Lemma~\ref{L8} shows that $H$ is finitely generated. In particular if $M \leq H$, then $H/M$ is finitely generated. In other words, every proper subgroup of $G/M$ is finitely generated. By Proposition~2.7 of the paper~\cite{KMO2008}, $G/M$ is a quasicyclic $q$-group for some prime $q$.

Let $L/M$ be a proper subgroup of $G/M$, then $L/M$ is a finite cyclic subgroup. An application of Lemma~\ref{L8} shows that  the subgroup $L$ is finitely generated. The finiteness of index  $|L : M|$  implies that $M$ is finitely generated (see, for example,~\cite[Corollary 7.2.1]{HM1959}).  Using Lemma~\ref{L7} we obtain that $M$ is soluble-by-finite. Let $S$ be a maximal normal soluble subgroup of $M$, then $M/S$ is finite. It follows that $S$ is $G$-invariant. Let $D = [S, \ S]$.  The factor-group  $M/D$ is abelian-by-finite and finitely generated, therefore it is noetherian. Let $V/D$  be a proper subgroup of $G/D$. If $M/D$ does not include $V/D$, then $M$ does not include $V$, and as above, $V$ is finitely generated. Then $V/D$ is also finitely generated. If $V/D \leq  M/D$, then $V/D$ is finitely generated too. Thus every proper subgroup of $G/D$ is finitely generated, and application of Proposition 2.7 of the paper~\cite{KMO2008}  shows that $G/D$ is a quasicyclic group. Since $M/D$ is a proper subgroup of $G/D$, $M/D$ is a finite cyclic subgroup. Suppose that $D \neq \langle1\rangle$, then $K = [D, D] \neq D$. Repeating the above arguments, we obtain that  $G/K$ is a quasicyclic group. In particular, it is abelian. Then $S/K$ is abelian, which follows that $K \geq [S, S] = D$, and we obtain a contradiction. This contradiction shows that $D = \langle1\rangle$, so that $G$ is a quasicyclic group.

\end{proof}

\begin{lemma}
\label{L2} Let $G$ be a locally generalized radical group and $A$ be an
minimax-antifinitary $\mathbb{Z}G$-module. Suppose that $G \neq \mathbf{Coc}%
_{\mathbb{Z}-mmx}(G)$, $G$ is not finitely generated, and $G/\mathbf{Coc}_{%
\mathbb{Z}-mmx}(G)$ is finitely generated. Then $G$ is soluble and $G/%
\mathbf{Coc}_{\mathbb{Z}-mmx}(G)$ is a group of a prime order $p$.
\end{lemma}

\begin{proof}
Again suppose that $C_G(A) = \langle1\rangle$. Let $D = \mathbf{Coc}_{\mathbb{Z}-mmx}(G)$. Since $G/D$ is finitely generated, $G=\langle M, D \rangle$ for some finite subset $M$. We may suppose that $M$ is minimal finite set with this property, that is  $G \neq \langle S, D \rangle$  for each proper subset $S$ of $M$. Suppose that $| M | \geq 2$. Then $M$ includes two proper subsets  $X$, $Y$ such that $M = X \cup Y$. By the choice of $M$, the subgroups $\langle X, D \rangle$ and  $\langle Y, D \rangle$  are proper and $\langle X, D \rangle \neq D$, $\langle Y, D \rangle \neq D$. By Lemma~\ref{L8} both subgroups $\langle X, D \rangle$ and $\langle Y, D\rangle$ are finitely generated. The equation $X \cup Y = M$ implies that $G = \langle X, Y, D \rangle$ is finitely generated. This contradiction shows that $|M| = 1$. In other words, $G/D$ is cyclic. Suppose that  $|G/D|$ is not a prime. Then $G$  includes a proper subgroup $B$ such that $D \leq B$, $B \neq D$,  and $G/B$ has a prime order. Using again Lemma~\ref{L8} we obtain that $B$ is finitely generated. The finiteness of $G/B$ follows that $G$ is finitely generated. This final contradiction proves that  $G/D$ has a prime order. Choose an element $g$ such that $G = \langle g, D \rangle$.

Since $G$ is not finitely generated, $D$ can not be finitely generated. Using Lemma~\ref{L7}, we obtain that $D$ is soluble-by-finite. Let $S$ be a maximal normal soluble subgroup of $D$ having finite index. Suppose that $D \neq S$. Clearly $S$ is $G$-invariant. Since  $D/S$  is finite non-soluble subgroup, $S\langle g^p \rangle \neq D$. It follows that $S\langle g \rangle$ is a proper subgroup of $G$. Since $D$ does not include $S\langle g \rangle$, $S\langle g \rangle$  is finitely generated by Lemma~\ref{L8}. Then $S\langle g^p \rangle$ is finitely generated (see, for example,~\cite[Corollary 7.2.1]{HM1959}). Since index  $| D : S |$ is finite, $D$ is finitely generated, and we obtain a contradiction. This contradiction shows that $D$ is soluble. Then the entire group $G$ is soluble.

\end{proof}

\begin{corollary}
\label{C3.3} Let $G$ be a locally generalized radical group and $A$ a
minimax-antifinitary $\mathbb{Z}G$-module. Suppose that $G \neq \mathbf{Coc}%
_{\mathbb{Z}-mmx}(G)$, $G$ is not finitely generated and $G/\mathbf{Coc}_{%
\mathbb{Z}-mmx}(G)$ is finitely generated. Let $g$ be an element of $G$ with
the property $G = \langle g \rangle \mathbf{Coc}_{\mathbb{Z}-mmx}(G)$. If $H$
is a normal subgroup of $G$, having finite index, then $H\langle g \rangle =
G$. Moreover, $G/H$ is a $p$-group where $p=|G/\mathbf{Coc}_{\mathbb{Z}%
-mmx}(G)|$.
\end{corollary}

\begin{proof}
Let $D = \mathbf{Coc}_{\mathbb{Z}-mmx}(G)$. If we assume that $H\langle g \rangle$ is a proper subgroup of $G$, then the choice of $g$ yields that $D$ does not include $H\langle g \rangle$. By Lemma~\ref{L8}, $H\langle g \rangle$  is finitely generated. Since $H\langle g \rangle$  has finite index, $G$  must be finitely generated, so we obtain a contradiction. This contradiction shows that $H\langle g \rangle = G$.

Suppose that $\Pi(G/H) \neq \{p\}$. Let $P/H$ be a Sylow  $p$-subgroup of $G/H$. Then $P/H$ is a proper subgroup of $G/H$. Since $P$  has finite index, $P$  is not finitely generated. Then $A/C_A(P)$ is artinian-by-(finite rank). It follows that $P \leq D$. On the other hand, $G/D$ is a non-identity $p$-group, therefore $D$ can not include $P$. This contradiction proves that $G/H$ is a $p$-group.
\end{proof}

\begin{proposition}
\label{P2.4} Let $G$ be a locally generalized radical group and $A$ a
minimax-antifinitary $\mathbb{Z}G$-module. Suppose that $G \neq \mathbf{Coc}%
_{\mathbb{Z}-mmx}(G)$, $G$ is not finitely generated and $G/\mathbf{Coc}_{%
\mathbb{Z}-mmx}(G)$ is finitely generated. If $G/[G, G]$ is infinite, then $%
G = Q \times \langle g \rangle$ where $Q$ is a quasicyclic $q$-subgroup, $g$
is a $p$-element and $g^p \in \mathbf{Coc}_{\mathbb{Z}-mmx}(G)$, where $p$, $%
q$ are primes (not necessary different).
\end{proposition}

\begin{proof}
As usual we suppose that $C_G(A) = \langle1\rangle$. Let $D = \mathbf{Coc}_{\mathbb{Z}-mmx}(G)$. By Lemma~\ref{L2}, $G$ is soluble and $G/D$  is a group of a prime order $p$. Choose an element $g$ such that $G = \langle g, D \rangle$.  Put $K = [G, G]$. Then $K \leq D$. Suppose that $K\langle g \rangle = G$. Since  $G/K$  is infinite, from  $G/K = K\langle g \rangle/K \cong \langle g \rangle/(\langle g \rangle \cap K)$  we obtain that $gK$ has infinite order. Let $r_1$, $r_2$  be two distinct primes. Then   $K\langle g^{r_j} \rangle$  is a proper subgroup of $G$. Since it has finite index in $G$, $K\langle g^{r_j} \rangle$  is not finitely generated. It follows that $A/C_A(K\langle g^{r_j} \rangle)$ is minimax. It is true for every $j \in \{1, 2\}$. The equation $r_1 \neq r_2$ implies that $\langle g \rangle = \langle g^{r_1} \rangle \langle g^{r_2} \rangle$. Corollary~\ref{C3} shows that  $A/C_A(\langle g \rangle)$  is minimax, that is $g \in D$, and we obtain a contradiction with the choice of $g$. This contradiction shows that $K\langle g \rangle$ is a proper subgroup of $G$.

Let $Z/(K\langle g \rangle)$ be a proper  subgroup of $G/(K\langle g \rangle)$. Then $D$ does not include $Z$ and Lemma~\ref{L8}  shows that $Z$ is finitely generated. If we assume that $Z$ has finite index in $G$, then $G$ must be finitely generated, so we obtain a contradiction. This contradiction shows that the factor-group $G/(K\langle g \rangle)$ is $\mathfrak{F}$-perfect. Then $G/(K\langle g \rangle)$  includes a subgroup  $P/(K\langle g \rangle)$ such that $G/P$ is a quasicyclic $q$-group for some prime $q$.  A subgroup  $P$ contains an element $g$, therefore $D$ does not include $P$. By Lemma~\ref{L8}, $P$  is finitely generated. It follows that  $G/K$  is an abelian minimax group. Suppose that $\mathbf{Tor}(G/K) \neq G/K$. Then $T/K = \mathbf{Tor}(D/K) \neq D/K$.  Put
$$\pi  = \{ r \ | \ r \mbox{ is a prime such that }  D/T \neq (D/T)^r \}.$$
Since $D/T$ is torsion-free and minimax, the set  $\pi$  is infinite. Therefore we can choose a prime $r$ such that $r \neq p$ and $r \in \pi$. Let  $L/T = (D/T)^r$, then $D/L$ is a non-identity elementary abelian $r$-group. By the choice of  $L$, $\Pi(G/L) = \{ r, p \}$, and we obtain a contradiction with Corollary~\ref{C3.3}. This contradiction shows that $G/K$  is periodic. In this case, $P/K$ is finite, so that $G/K$ is a Chernikov group. Let $Q/K$ be a divisible part of $G/K$. An isomorphism $Q/K \cong G/P$ shows that $Q/K$ is a quasicyclic $q$-subgroup. Since $Q$ has finite index, then an application of Corollary~\ref{C3.3} shows that $G = Q\langle g \rangle$ and $G/Q$ is a $p$-group. It follows that $G/K = Q/K \times \langle gK \rangle$ (see, for example,~\cite[Theorem 21.2]{FL1970}).

Suppose that $K \neq \langle 1\rangle$. Then $L = [K, K] \neq K$. We have already proved above that $K\langle g \rangle$ is a proper subgroup of $G$. Since $D$ does not include $K\langle g \rangle$, Lemma~\ref{L8} shows that $K\langle g \rangle$  is finitely generated. The fact that $G/K$ is periodic implies that K  has finite index in $K\langle g \rangle$. Then $K$ is finitely generated (see, for example,~\cite[Corollary 7.2.1]{HM1959}).  Thus $K/L$ is a finitely generated abelian group. Then $K/L$ includes a proper  $G$-invariant subgroup $V/L$, having finite index in $K/L$ (this subgroup can be identity). Then  $G/V$  is a Chernikov group, having finite derived subgroup. Let  $Q_1/V$  be the divisible part of  $G/V$, then $Q_1/V \cong Q/K$, so that  $Q_1$/$V$  is a quasicyclic  $q$-subgroup. Since $[G/V, G/V]$ is finite, $Q_1/V \leq \zeta(G/V)$. Since index  $| G : Q_1 |$ is finite, $G = Q_1\langle g \rangle$ by Corollary~\ref{C3.3}. This equation together with inclusion $Q_1/V \leq \zeta(G/V)$ implies that $G/V$  is abelian. But in this case $K \leq V$, and we obtain a contradiction with the choice of $V$. This contradiction proves that $K = \langle 1 \rangle$.
\end{proof}

\begin{proposition}
\label{P2.5} Let $G$ be a locally generalized radical group and $A$ a
minimax-antifinitary $\mathbb{Z}G$-module. Suppose that $G \neq \mathbf{Coc}%
_{\mathbb{Z}-mmx}(G)$, $G$ is not finitely generated and $G/\mathbf{Coc}_{%
\mathbb{Z}-mmx}(G)$ is finitely generated. If $G/[G, G]$ is finite, then $G$
includes a normal divisible Chernikov $q$-subgroup $Q$, such that $G =
Q\langle g \rangle$ where $g$ is a $p$-element, $g^p \in \mathbf{Coc}_{%
\mathbb{Z}-mmx}(G)$ and $p$, $q$ are primes (not necessary different).
Moreover, a subgroup $Q $ is $G$-quasifinite.
\end{proposition}

\begin{proof}
As usual we suppose that $C_G(A) = \langle 1 \rangle$. Let $D = \mathbf{Coc}_{\mathbb{Z}-mmx}(G)$. By Lemma~\ref{L2}, $G$ is soluble and $G/D$ is a group of a prime order $p$. Choose an element $g$ such that $G = \langle g, D \rangle$.  Put $K = [G, G]$. Since $G/K$  is finite, Corollary~\ref{C3.3} shows that $G = K\langle g \rangle$ and $G/K$ is a $p$-group. It follows that $K$ is not finitely generated.

Since $G$ is not finitely generated and soluble, $L = [K, K]$ is a proper subgroup of $K$. If we suppose that $\langle g, L \rangle = G$, then $G/L = \langle g \rangle L/L \cong \langle g \rangle/(\langle g \rangle \cap L)$ is abelian. It follows that $K \leq L$, and we obtain a contradiction. Thus  $\langle g, L \rangle$ is a proper subgroup of $G$. If we suppose that $G/L$ is finite, then Corollary~\ref{C3.3} shows that $G = L\langle g \rangle$. Hence  $G/L$ is infinite, i.e. $K/L$ is infinite. As we noted above, $\langle L, g \rangle$ is a proper subgroup of $G$. Since $D$ does not include         $\langle L, g \rangle$, $\langle L, g \rangle$ is finitely generated by Lemma~\ref{L8}.  Put $\langle v \rangle = \langle g \rangle \cap K$. We have
$$K \cap (L \langle g \rangle) = L(K \cap \langle g \rangle) = L \langle v \rangle.$$
Clearly $L\langle v \rangle$ is a  $G$-invariant subgroup of $K$. Furthermore, $|\langle L, g \rangle : L\langle v \rangle | \leq |G : D| = p$. It follows that $\langle L, v \rangle$ is finitely generated (see, for example,~\cite[Corollary 7.2.1]{HM1959}). If we suppose that $K/(L\langle v \rangle)$ is finitely generated, then $K$ is finitely generated, and we obtain a contradiction. This contradiction shows that $K/(L\langle v \rangle)$ is not finitely generated.

Let $Z/(L\langle v \rangle)$ be a proper $G$-invariant subgroup of $K/(L\langle v \rangle)$. We have $Z\langle g \rangle \cap K = X(\langle g \rangle \cap K) = Z\langle v \rangle = Z$. It follows that $Z\langle g \rangle$ is a proper subgroup of $G$. Since $D$ does not include $Z\langle g \rangle$, $Z\langle g \rangle$  is finitely generated by Lemma~\ref{L8}.

Assume that $K/(L\langle v \rangle)$ includes a proper subgroup $U/\langle L, v \rangle$, having finite index. Then $|G : U|$  is finite, so that  $U_1 = \mathbf{Core}_G(U)$ has finite index in $G$. By above proved $U_1\langle g \rangle$  is finitely generated. Finiteness of $| G : U_1|$ implies that $G$ is finitely generated. This contradiction shows that $K/(L\langle v \rangle)$ is $\mathfrak{F}$-perfect. Then $K/(L\langle v \rangle)$ includes a subgroup $P/(L\langle v \rangle)$ such that $K/P$ is a quasicyclic $q$-group for some prime $q$.  We remark that $K/P^x = K^x/P^x \cong K/P$, i.e. $K/P^x$ is a quasicyclic $q$-group for all $x \in G$. Finiteness of $G/K$ implies that a family  $\{ P^x \ | \ x \in G \} $ is finite. Let $\{ P^x \ | \ x \in G \} = \{ P_1 = P, P_2, \ldots , P_m \}$. By Remak's theorem we obtain the embedding
$$K/\mathbf{Core}_G(P) \hookrightarrow G/P_1 \times G/P_2 \times \ldots \times G/P_m,$$
which shows that $K/\mathbf{Core}_G(P)$ is a Chernikov $q$-group. Being $\mathfrak{F}$-perfect, it is divisible. Since $\langle L, v \rangle \leq P$ and $\langle L, v \rangle$ is $G$-invariant, $\langle L, v \rangle \leq C = \mathbf{Core}_G(P)$.  By proved above, $C$ is  finitely generated. In particular, $C/L$ is an abelian finitely generated group, so that $K/L$ is an abelian minimax group. Suppose that $\mathbf{Tor}(K/L) = T/L \neq K/L$. Put
$$\pi = \{ r \ | \   r \mbox{ is a prime such that } K/T \neq (K/T)^r \}.$$
Since $K/T$ is torsion-free and minimax, the set $\pi$  is infinite. Therefore we can choose a prime $r$ such that $r \neq p$ and $r \in \pi$. Let  $M/T = (K/T)^r$, then $K/M$ is a non-identity elementary abelian $r$-group. Clearly a subgroup $M$ is $G$-invariant. By the choice of $M$, $\Pi(G/M) = \{ r, p \}$, and we obtain a contradiction with Corollary~\ref{C3.3}. This contradiction shows that $K/L$ is periodic. In this case,  $C/L$ is finite, so that $K/L$ is Chernikov. Let $Q/L$  be a divisible part of $K/L$. An isomorphism $Q/L \cong K/C$ shows that $Q/L$ is a $q$-subgroup. Since $Q$ has finite index, then an application of Corollary~\ref{C3.3} shows that $G = Q\langle g \rangle$ and $G/Q$ is a $p$-group.

Suppose that $Q/L$ includes an infinite $G$-invariant subgroup $Q_1/L$. Suppose that $Q_1\langle g \rangle$  is finitely generated. Then           $Q_1\langle g \rangle/L = (Q_1/L)\langle gL \rangle$ is also finitely generated. A factor-group $G/L$ is periodic, in particular, $\langle gL \rangle$  is finite. It follows that $Q_1/L$ is finitely generated. On the other hand, $Q_1/L$ is an infinite Chernikov group, therefore it can not be finitely generated. This contradiction shows that $Q_1\langle g \rangle$ is not finitely generated. Then $A/C_A(Q_1\langle g \rangle)$ is artinian-by-(finite rank). Corollary~\ref{C3} shows that $g \in D$, and we obtain a contradiction. This contradiction shows that $Q/L$ is  $G$-quasifinite.

Suppose that $L \neq \langle 1 \rangle$. Then $V = [L, L] \neq L$. We have already proved that $L\langle g \rangle$ is is finitely generated. The fact that $G/L$ is periodic implies that $L$ has finite index in $L\langle g \rangle$. Then $L$ is finitely generated (see, for example,~\cite[Corollary 7.2.1]{HM1959}).  Thus $L/V$ is a finitely generated abelian group. Then $L/V$ includes a proper $G$-invariant subgroup $W/V$, having finite index in $L/V$ (this subgroup can be identity). Then $K/W$ is a Chernikov group, having finite derived subgroup. Let $Q_2/W$ be the divisible part of $K/W$, then  $Q_2/W \cong Q/L$, so that $Q_2/W$ is a divisible Chernikov $q$-subgroup. Since $[K/W, K/W]$ is finite, $Q_2/W \leq \zeta(K/W)$. Since index $| G : Q_2 |$ is finite, $G = Q_2\langle g \rangle$ by Corollary~\ref{C3.3}. Then
$$K = K \cap (Q_2\langle g \rangle) = Q_2(K \cap \langle g \rangle) =  Q_2\langle v \rangle.$$
It follows that $K/Q_2$ is cyclic. Then the inclusion $Q_2/W \leq \zeta(K/W)$ implies that $K/W$ is abelian. But in this case $L \leq W$, and we obtain a contradiction with the choice of $W$. This contradiction proves that $L$ is abelian.
\end{proof}

Recall that a group $G$ is said to have \textit{finite special rank} $%
\mathbf{r}(G) = r$ if every finitely generated subgroup of $G$ has at most $r
$ generators and there exists a finitely generated subgroup $H$ of $G$ such
that $H$ has exactly $r$ generators. We remark that every abelian minimax
group has finite special rank.

\begin{lemma}
\label{L3.6} Let $G$ be a Chernikov group and $A$ a $\mathbb{Z}G$-module. If 
$A/C_A(G)$ is minimax (as a $\mathbb{Z}$-module), then the additive group of 
$A(\omega\mathbb{Z}G)$ is a Chernikov subgroup. Moreover, $\Pi(A(\omega%
\mathbb{Z}G)) \subseteq \Pi(G)$.
\end{lemma}

\begin{proof}
For each element $x$ of $G$ consider the mapping $\delta_x: A \mapsto A$, defined by the rule $\delta_x(a) = a(x-1)$, $a \in A$. Clearly this mapping is a $\mathbb{Z}$-endomorphism of $A$, $\mathbf{Ker}(\delta_x) = C_A(x)$ and $\mathbf{Im}(\delta_x) = A(\omega\mathbb{Z}\langle x \rangle) = A(x-1)$.  Hence $A(x-1) = \mathbf{Im}(\delta_x) \cong_\mathbb{Z} A/\mathbf{Ker}(\delta_x) = A/C_A(x)$. Being minimax, $A/C_A(G)$ has finite special rank $r$. The inclusion $C_A(G) \leq C_A(x)$ follows that $A/C_A(x)$ has a special rank at most $r$. Then and $\mathbf{r}(A(x-1))\leq r$. Let $k$ be a positive integer such that  $|\Omega_1(G) | = q^k$. Then $G$ has an ascending series of finite subgroups
$$L_1 = \Omega_1(G) \leq L_2 \leq \ldots \leq L_n \leq L_{n + 1} \leq \ldots$$
such that $L_n = \mathbf{Dr}_{1 \leq j \leq k} \langle x_{n_j} \rangle$  where $|x_{n_j}| \leq q^k$ for each $j$, and $G = \bigcup_{n \in \mathbb{N}} L_n$. The equation
$$A(\omega\mathbb{Z}L_n) = A(\omega\mathbb{Z}\langle x_{n_1} \rangle) + \ldots + A(\omega\mathbb{Z}\langle x_{n_k} \rangle) =  A(x_{n_1}-1) + \ldots + A(x_{n_k}-1)$$
together with $\mathbf{r}(A(x_{n_j}-1)) \leq r$, $1 \leq j \leq k$, shows that $\mathbf{r}(A(\omega\mathbb{Z}L_n)) \leq rk$, $n \in \mathbb{N}$. Since $G = \bigcup_{n \in \mathbb{N}} L_n$, $A(\omega\mathbb{Z}G) = \bigcup_{n \in \mathbb{N}} A(\omega\mathbb{Z}L_n)$. Moreover  $L_n \leq L_{n + 1}$ implies that $A(\omega\mathbb{Z}L_n) \leq A(\omega\mathbb{Z}L_{n + 1})$. Let $B$  be an arbitrary finitely generated subgroup of  $A(\omega\mathbb{Z}G)$. Then there exists a positive integer $m$ such that $B \leq A(\omega\mathbb{Z}L_m)$. By proved above $B$ has at most $rk$ generators. It follows that $A(\omega\mathbb{Z}G)$ has a finite special rank at most $rk$.

Let $Q$ be a divisible part of $G$. Since $A/C_A(Q)$ is minimax, $A$ has a series of $\mathbb{Z}G$-submodules $C_G(Q) = C \leq T \leq A$ where $T/C = \mathbf{Tor}(A/C)$ is a Chernikov group and $A/T$ is torsion-free and has finite $\mathbb{Z}$-rank. Repeating the final part of the proof of Lemma~\ref{L5}, we obtain that $Q = C_Q(T)$  and  $Q = C_Q(A/T)$.

Let $a$ be an arbitrary element of $T$. Consider the mapping $\gamma_a: Q \mapsto A(\omega\mathbf{Z}Q)$, defined by the rule $\gamma_a(x) = a(x-1)$. By $(x-1)(y-1) = (xy-1)-(x-1)-(y-1)$. We have $a(xy-1) = a(x-1) + a(y-1) + a(x-1)(y-1) = a(x-1) + a(y-1)$. The equation $Q = C_Q(T)$ implies that  $a(x-1)(y-1) = 0$. In other words, $\gamma_a(xy) = \gamma_a(x) + \gamma_a(y)$, thus $\gamma_a$  is a homomorphism of $G$  in $A(\omega\mathbb{Z}Q)$. Furthermore, $\mathbf{Ker}(\gamma_a) = C_G(a)$  and $\mathbf{Im}(\gamma_a) = \langle a \rangle(\omega\mathbb{Z}Q) = [a, Q]$, so that $[a, Q] \cong Q/C_Q(a)$. It follows that if $[a, Q] \neq \langle 0 \rangle$, then it is a divisible Chernikov subgroup and  $\Pi([a, Q]) \subseteq \Pi(Q)$. Since it is valid for every $a \in T$, $T(\omega\mathbb{Z}Q)$ is a divisible subgroup (if it is non-zero) and $\Pi(T(\omega\mathbb{Z}Q)) \subseteq \Pi(Q)$. By proved above, $T(\omega\mathbb{Z}Q)$ has finite special rank, and therefore $T(\omega\mathbb{Z}Q)$ is a Chernikov subgroup.

Consider now the factor-module $A/V$ where $V = T(\omega\mathbb{Z}Q)$. Then the inclusion  $T/V \leq C_{A/V}(Q)$  implies  that $(A/V)(\omega\mathbb{Z}Q) \leq T/V$. Using the above arguments, we obtain that $ (A/V)(\omega\mathbb{Z}Q)$ is a Chernikov divisible group such that  $\Pi((A/V)(\omega\mathbb{Z}Q)) \subseteq \Pi(Q)$. We have
$$(A/V)(\omega\mathbb{Z}Q)  = (A(\omega\mathbb{Z}Q) + V)/V = (A(\omega\mathbb{Z}Q) + T(\omega\mathbb{Z}Q))/(T(\omega\mathbb{Z}Q),$$
which follows that $A(\omega\mathbb{Z}Q)$ is a Chernikov divisible  subgroup such that $\Pi(A(\omega\mathbb{Z}Q)) \subseteq \Pi(Q)$.

Let $M = A(\omega\mathbb{Z}Q)$, then $Q \leq C_G(A/M)$, in particular, $G/C_G(A/M)$ is finite. By proved above $(A/M)(\omega\mathbb{Z}G)$ has finite special rank. Using the above arguments, we obtain that $\langle a + M \rangle(\omega\mathbb{Z}G)$ is a finite group and                   $\Pi(\langle a+M \rangle(\omega\mathbb{Z}G)) \subseteq \Pi(G)$  for every element $a \in A$. The finiteness of $\Pi(G)$ implies that $(A/M)(\omega\mathbb{Z}G)$ is a Chernikov subgroup and $\Pi((A/M)(\omega\mathbb{Z}G)) \subseteq \Pi(G)$. Hence $A(\omega\mathbb{Z}G)$ is Chernikov and $ \Pi(A(\omega\mathbb{Z}G)) \subseteq  \Pi(G)$.

\end{proof}

\section{Proofs of the main Theorem.}

\begin{proof}

If $G/D$ is not finitely generated, then Proposition~\ref{P2.1} shows that $G$ is a group of type (1).

Suppose now that $G/D$ is finitely generated. Then Lemma~\ref{L2}  proves that $G$  is soluble and $G/D$ is a group of a prime order $p$. If we assume that  $G/[G, G]$ is infinite, then Proposition~\ref{P2.4} shows that $G$ is a group of type (2).

Finally suppose that $G/[G, G]$ is finite. Then Proposition~\ref{P2.5} shows that  G  includes a normal divisible Chernikov $q$-subgroup $Q$, such that $G = Q\langle g \rangle$ where $g$ is  a $p$-element, $p$, $q$ are primes (not necessary different). Moreover, $g^p \in D$  and $Q$ is $G$-quasifinite. Finally, the assertion~\ref{3c} follows from the results of Section 3 of the paper~\cite{ZD1974}, and the assertion~\ref{3d}  follows from Theorem 3.4 of the paper~\cite{HB1977}.

Let $G$ be a group of the type (2) or (3). Then $D = Q\langle g^p \rangle$  is a proper Chernikov subgroup of $G$, and hence it is not finitely generated. Then $A/C_A(D)$ is minimax and we can use Lemma~\ref{L3.6}.

\end{proof}

We will construct the following example showing that situations described in
the main theorem are real.

In the case (1) $G$ is a quasicyclic $q$-group and $\mathbf{Coc}_{\mathbb{Z}%
-mmx} (G) = \langle 1 \rangle$. Let $p$ be a prime such that $p \neq q$ and $%
A$ be a simple $\mathbf{F_p}G$-module. The method of constructing such a
module is specified, for example, in Chapter 2 of the book~\cite{KOS2002}.
We can consider $A$ as a $\mathbb{Z}G$-module. For this module we have $%
C_A(g) = \langle 0 \rangle$ for each element $g$ of $G$. It follows that $%
\mathbf{Coc}_{\mathbb{Z}-mmx} (G) = \langle 1 \rangle$.

Let $A = \langle a \rangle$ be an additively written infinite cyclic group
and $B = \langle b_n \ | \ n \in \mathbb{N} \rangle$ additively written
quasicyclic $2$-group, that is $2b_1 = 0$, $2b_2 = b_1$, \ldots, $2b_{n + 1}
= b_n$, $n \in \mathbb{N}$. Put $D = A \oplus B$.

Let $\gamma_k$ be an automorphism of $D$, satisfying the following
conditions:

\begin{equation*}
\gamma_k(a) = a + b_k, \ \gamma_k(b_n) = b_n \mbox{ for all } n \in \mathbb{N%
}.
\end{equation*}

Then $\gamma_1^2 = \varepsilon$ is an identity automorphism of $D$, $%
\gamma_2^2 = \gamma_1$ and $\gamma_{k + 1}^2 = \gamma_k$, $k \in \mathbb{N}$%
. In other words, $\Gamma = \langle \gamma_k \ | \ k \in \mathbb{N} \rangle$
is a quasicyclic $2$-group. In a natural way, $D$ becomes a $\mathbb{Z}\Gamma
$-module. Furthermore, $C_D(\gamma_k) = C_D(\Gamma) = B$, in particular, $%
D/C_D(\Gamma)$ is an infinite cyclic group.

Let $\langle g \rangle$ be a cyclic group of a prime order $p$, $p \neq 2$.
For every prime $q \notin \{2, p\}$ there exists a finite simple $\mathbf{F_q%
}\langle g \rangle$-module $U$. Put $V =\bigoplus_{j \in \mathbb{N}} V_j$,
where $V_j$ is an $\mathbf{F_q}\langle g \rangle$-isomorphic copy of $U$, $j
\in \mathbb{N}$. Then $V$ is an infinite $\mathbf{F_q}\langle g \rangle$%
-module such that $C_V(g) = \langle 0 \rangle$. Again we can consider $V$ as 
$\mathbb{Z}\langle g \rangle$-module. Furthermore, put $E = D \oplus V$. We
can define an action of $G = \Gamma \times \langle g \rangle$ on $E$ such
that $\Gamma$ acts trivially on $V$ and $g$ acts trivially on $D$. Thus $E$
becomes a $\mathbb{Z}G$-module. Furthermore, $E/C_E(\Gamma) = E/(B \oplus V)$
is an infinite cyclic group and $E/C_E(g) = E/D$ is an infinite elementary
abelian $q$-subgroup. Hence $\mathbf{Coc}_{\mathbb{Z}-mmx} (G) = \Gamma$,
and we obtain the situation described in (2).

Consider again the above constructed $\mathbb{Z}\Gamma$-module $D$. Let $T$
be an infinite elementary abelian $3$-group. We can extend an action of on $%
Y = D \oplus T$ if put $\gamma_k(c) = c$ for each element $c \in T$ and $k
\in \mathbb{N}$. Define an automorphism $\delta$ of $Y$ by the rule $%
\delta(a) = a$, $\delta(c) =-c$ for each element $c \in B \oplus T$. Then $%
\delta^2 = \varepsilon$. Further, we have 
\begin{equation*}
(\delta \circ \gamma_k \circ \delta)(a) = \delta(\gamma_k(\delta(a)) =
\delta(\gamma_k(a)) = \delta(a + b_k) = a - b_k = \gamma_k^{-1}(a),
\end{equation*}
\begin{equation*}
(\delta \circ \gamma_k \circ \delta)(c) = \delta(\gamma_k(\delta(c)) =
\delta(\gamma_k(-c)) = \delta(-c) = c = \gamma_k^{-1}(c), \mbox{ whenever }
c \in B \oplus T.
\end{equation*}
It follows that $\delta \circ \gamma_k \circ \delta = \gamma_k^{-1}$, $k \in 
\mathbb{N}$. In particular, $\Gamma$ is $\langle \delta \rangle$-invariant
and $G = \Gamma\langle \delta \rangle$ is an infinite dihedral $2$-group.
Now we can consider $Y$ as $\mathbb{Z}G$-module. For this module we have $%
Y/C_Y(\Gamma) = Y/(B \oplus T)$ is an infinite cyclic group and $%
Y/C_Y(\delta) = E/D$ is an infinite elementary abelian $3$-subgroup. Again $%
\Gamma = \mathbf{Coc}_{\mathbb{Z}-mmx} (G)$, and we obtain the situation
described in (3).

\end{document}